\documentclass[11pt]{article}
\usepackage[english]{babel}
\usepackage{amssymb,amsmath,amsthm}
\newtheorem{theorem}{Theorem}[section]
\newtheorem{lemma}[theorem]{Lemma}
\newtheorem{proposition}[theorem]{Proposition}

{\theoremstyle{definition}}
{\theoremstyle{definition}\newtheorem{example}[theorem]{Example}}
{\theoremstyle{definition}}
{\theoremstyle{definition}\newtheorem{remark}[theorem]{Remark}}

\newtheorem*{thma}{Theorem A}
\newtheorem*{exb}{Example B}
\newtheorem*{prc}{Proposition C}

\numberwithin{equation}{section}

\def\C{{\mathbb C}}

\def\N{{\mathbb N}}
\def\Z{{\mathbb Z}}
\def\R{{\mathbb R}}

\def\T{{\mathbb T}}
\def\K{{\mathbb K}}
\def\F{{\mathcal F}}

\def\uu{{\mathcal U}}
\def\epsilon{\varepsilon}
\def\kappa{\varkappa}
\def\phi{\varphi}
\def\leq{\leqslant}
\def\geq{\geqslant}

\def\dim{\hbox{\tt dim}\,}
\def\ker{\hbox{\tt ker}\,}
\def\spann{\hbox{\tt span}\,}

\title{On supercyclicity of operators from a supercyclic semigroup}

\author{Stanislav Shkarin}

\date{}

\begin{document}

\maketitle

\begin{abstract}
We show that for every supercyclic strongly continuous operator
semigroup $\{T_t\}_{t\geq 0}$ acting on a complex $\F$-space, every
$T_t$ with $t>0$ is supercyclic. Moreover, the set of supercyclic
vectors of each $T_t$ with $t>0$ is exactly the set of supercyclic
vectors of the entire semigroup.
\end{abstract}

\small \noindent{\bf MSC:} \ \ 47A16, 37A25

\noindent{\bf Keywords:} \ \ Hypercyclic semigroups, hypercyclic
operators, supercyclic operators, supercyclic semigroups \normalsize

\section{Introduction \label{s1}}\rm

Unless stated otherwise, all vector spaces in this article are over
the field $\K$, being either the field $\C$ of complex numbers or
the field $\R$ of real numbers and all topological spaces {\it are
assumed to be Hausdorff}. As usual, $\Z_+$ is the set of
non-negative integers, $\N$ is the set of positive integers and
$\R_+$ is the set of non-negative real numbers. The symbol $L(X)$
stands for the space of continuous linear operators on a topological
vector space $X$, while $X'$ is the space of continuous linear
functionals on $X$. As usual, for $T\in L(X)$, the dual operator
$T':X'\to X'$ is defined by the formula $T'f(x)=f(Tx)$ for $x\in X$
and $f\in X'$. Recall that an {\it affine} map $T$ on a vector space
$X$ is a map of the shape $Tx=u+Sx$, where $u$ is a fixed vector in
$X$ and $S:X\to X$ is linear. Clearly, $T$ is continuous if and only
if $S$ is continuous. The symbol $A(X)$ stands for the space of
continuous affine maps on a topological vector space $X$. An
$\F$-space is a complete metrizable topological vector space. Recall
that a family $\F=\{T_a\}_{a\in A}$ of continuous maps from a
topological space $X$ to a topological space $Y$ is called {\it
universal} if there is $x\in X$ for which $\{T_ax:a\in A\}$ is dense
in $Y$ and such an $x$ is called a {\it universal element} for $\F$.
We use the symbol $\uu({\cal F})$ for the set of universal elements
for $\cal F$. If $X$ is a topological space and $T:X\to X$ is a
continuous map, then we say that $x\in X$ is {\it universal} for $T$
if $x$ is universal for the family $\{T^n:n\in\Z_+\}$. We denote the
set of universal elements for $T$ by $\uu(T)$. A family
$\F=\{T_t\}_{t\in\R_+}$ of continuous maps from a topological space
$X$ to itself is called a {\it semigroup} if $T_0=I$ and
$T_{t+s}=T_tT_s$ for every $t,s\in\R_+$. We say that a semigroup
$\{T_t\}_{t\in\R_+}$ is {\it strongly continuous} if $t\mapsto T_t
x$ is continuous as a map from $\R_+$ to $X$ for every $x\in X$ and
we say that $\{T_t\}_{t\in\R_+}$ is {\it jointly continuous} if
$(t,x)\mapsto T_tx$ is continuous as a map from $\R_+\times X$ to
$X$. If $X$ is a topological vector space, we call a semigroup
$\{T_t\}_{t\in\R_+}$ a {\it linear semigroup} if $T_t\in L(X)$ for
every $t\in\R_+$ and $\{T_t\}_{t\in\R_+}$ is called an {\it affine
semigroup} if $T_t\in A(X)$ for every $t\in\R_+$. Recall that $T\in
L(X)$ is called {\it hypercyclic} if $\uu(T)\neq\varnothing$ and
elements of $\uu(T)$ are called {\it hypercyclic vectors}. A
universal linear semigroup $\{T_t\}_{t\in\R_+}$ is called {\it
hypercyclic} and its universal elements are called {\it hypercyclic
vectors} for $\{T_t\}_{t\in\R_+}$. If $T\in L(X)$, then universal
elements of the family $\{zT^nx:z\in\K,\ n\in\Z_+\}$ are called {\it
supercyclic vectors} for $T$ and $T$ is called {\it supercyclic} if
it has a supercyclic vector. Similarly, if $\{T_t\}_{t\in\R_+}$ is a
linear semigroup, then a universal element of the family
$\{zT_t:z\in\K,\ t\in\R_+\}$ is called a {\it supercyclic vector}
for $\{T_t\}_{t\in\R_+}$ and the semigroup is called {\it
supercyclic} if it has a supercyclic vector.

Hypercyclicity and supercyclicity have been intensely studied during
the last few decades, see \cite{bama-book} and references therein.
Our concern is the relation between the supercyclicity of a linear
semigroup and supercyclicity of the individual members of the
semigroup. The hypercyclicity version of the question was treated by
Conejero, M\"uller and Peris \cite{semi}, who proved that for every
strongly continuous hypercyclic linear semigroup
$\{T_t\}_{t\in\R_+}$ on an $\F$-space, each $T_s$ with $s>0$ is
hypercyclic and $\uu(T_s)=\uu(\{T_t\}_{t\in\R_+})$. Virtually the
same proof works in the following much more general setting
\cite[Chapter~3]{bama-book}.

\begin{thma}Let $\{T_t\}_{t\in\R_+}$ be a hypercyclic
jointly continuous linear semigroup on any topological vector space
$X$. Then each $T_s$ with $s>0$ is hypercyclic and
$\uu(T_s)=\uu(\{T_t\}_{t\in\R_+})$.
\end{thma}

The stronger condition of joint continuity coincides with the strong
continuity in the case when $X$ is an $\F$-space due to a
straightforward application of the Banach--Steinhaus theorem. The
essential part of the proofs in \cite{semi,bama-book} does not
really need linearity. It is based on a homotopy-type argument and
goes through without any changes (under certain assumptions) for
semigroups of non-linear maps. Recall that a topological space $X$
is called {\it connected} if it has no subsets different from
$\varnothing$ and $X$, which are closed and open and it is called
{\it simply connected} if for any continuous map $f:\T\to X$, there
is a continuous map $F:\T\times [0,1]\to X$ and $x_0\in X$ such that
$F(z,0)=f(z)$ and $F(z,1)=x_0$ for any $z\in\T$. Next, $X$ is called
{\it locally path connected at} $x\in X$  if for any neighborhood
$U$ of $x$, there is a neighborhood $V$ of $x$ such that for any
$y\in V$, there is a continuous map $f:[0,1]\to X$ satisfying
$f(0)=x$, $f(1)=y$ and $f([0,1])\subseteq U$. A space $X$ is called
{\it locally path connected} if it is locally path connected at
every point. Just listing the conditions needed for the proof in
\cite{semi,bama-book} to run smoothly, we get the following result.

\begin{proposition}\label{gege} Let $X$ be a topological space and
$\{T_t\}_{t\in\R_+}$ be a jointly continuous semigroup on $X$ such
that
\begin{itemize} \itemsep=-2pt
\item[\rm(1)] $\{T_tu:t\in[0,c]\}$ is nowhere dense in $X$ for every $c>0$ and
$u\in X;$
\item[\rm(2)] for every $c>0$ and $x\in\uu(\{T_t\}_{t\in\R_+})$,
there is $Y_{c,x}\subseteq X$ such that $Y_{c,x}$ is connected,
locally path connected, simply connected and
$\{T_tx:t\in[0,c]\}\subseteq Y_{c,x}\subseteq
\uu(\{T_t\}_{t\in\R_+})$.
\end{itemize}
Then $\uu(T_s)=\uu(\{T_t\}_{t\in\R_+})$ for every $s>0$.
\end{proposition}

The natural question whether the supercyclicity version of Theorem~A
holds was touched by Bernal-Gonz\'alez and Grosse-Erdmann in
\cite{ex2}. They have produced the following example.

\begin{exb} Let $X$ be a Banach space over $\R$,
$\{T_t\}_{t\in\R_+}$ be a hypercyclic linear semigroup on $X$ and
$A_t\in L(\R^2)$ for $t\in\R_+$ be the linear operator with the
matrix
$A=\Bigl(\begin{matrix}\scriptstyle\cos t&\scriptstyle\sin t\\
\noalign{\kern-3pt} \scriptstyle-\sin t&\scriptstyle\cos
t\end{matrix}\Bigr)$. Then $\{A_t\oplus T_t\}_{t\in\R_+}$ is a
supercyclic linear semigroup on $\R^2\times X$, while $A_t\oplus
T_t$ is non-supercyclic whenever $\frac{t}{\pi}$ is rational.
\end{exb}

Example~B shows that the natural supercyclicity version of Theorem~A
fails in the case $\K=\R$. In the complex case, the following
partial result was obtained by Bayart and Matheron
\cite[p.~73]{bama-book}.

\begin{prc}Let $X$ be a complex topological vector space and
$\{T_t\}_{t\in\R_+}$ be a supercyclic jointly  continuous linear
semigroup on $X$ such that $T_t-\lambda I$ has dense range for every
$t>0$ and every $\lambda\in\C$. Then each $T_t$ with $t>0$ is
supercyclic. Moreover, the set of supercyclic vectors for $T_t$ does
not depend on the choice of $t>0$ and coincides with the set of
supercyclic vectors of the entire semigroup.
\end{prc}

The argument in \cite{bama-book} is another adaptation of the proof
in \cite{semi}, however one can obtain the same result directly by
considering the induced action on subsets of the projective space
and applying Proposition~\ref{gege}. We will show that in the case
$\K=\C$, the supercyclicity version of Theorem~A holds without any
additional assumptions.

\begin{theorem}\label{main} Let $X$ be a complex topological vector
space and $\{T_t\}_{t\in\R_+}$ be a supercyclic jointly continuous
linear semigroup on $X$. Then each $T_s$ with $s>0$ is supercyclic
and the set of supercyclic vectors of $T_s$ coincides with the set
of supercyclic vectors of $\{T_t\}_{t\in\R_+}$.
\end{theorem}

It turns out that any supercyclic jointly continuous linear
semigroup on a complex topological vector $X$ either satisfies
conditions of Proposition~C or has a closed invariant hyperplane
$Y$. In the latter case the issue reduces to the following
generalization of Theorem~A to affine semigroups.

\begin{theorem}\label{BBB} Let $X$ be a topological vector space
and $\{T_t\}_{t\in\R_+}$ be a universal jointly continuous affine
semigroup on $X$. Then each $T_s$ with $s>0$ is universal and
$\uu(T_s)=\uu(\{T_t\}_{t\in\R_+})$.
\end{theorem}

\section{A dichotomy for supercyclic linear semigroups}

An analogue of the following result for individual supercyclic
operators is well-known \cite{bama-book}.

\begin{proposition}\label{dich} Let $X$ be a complex topological vector
space and $\{T_t\}_{t\in\R_+}$ be a supercyclic strongly continuous
linear semigroup on $X$. Then either $(T_t-\lambda I)(X)$ is dense
in $X$ for every $t>0$ and $\lambda\in\C$ or there is a closed
hyperplane $H$ in $X$ such that $T_t(H)\subseteq H$ for every
$t\in\R_+$.
\end{proposition}

The most of the section is devoted to the proof of
Proposition~\ref{dich}. We need several elementary lemmas. Recall
that a subset $B$ of a vector space $X$ is called {\it balanced} if
$\lambda x\in B$ for every $x\in B$ and $\lambda\in\K$ such that
$|\lambda|\leq 1$.

\begin{lemma}\label{tv1} Let $K$ be a compact subset of an infinite
dimensional topological vector space and $X$ such that $0\notin K$.
Then $\Lambda=\{\lambda x:\lambda\in\K,\ x\in K\}$ is a closed
nowhere dense subset of $X$.
\end{lemma}

\begin{proof} Closeness of $\Lambda$ in $X$ is a straightforward
exercise. Assume that $\Lambda$ is not nowhere dense. Since
$\Lambda$ is closed, its interior $L$ is non-empty. Since $K$ is
closed and $0\notin K$, we can find a non-empty balanced open set
$U$ such that $U\cap K=\varnothing$. Clearly $\lambda x\in L$
whenever $x\in L$ and $\lambda\in\K$, $\lambda\neq 0$. Since $U$ is
open and balanced the latter property of $L$ implies that the open
set $W=L\cap U$ is non-empty. Taking into account the definition of
$\Lambda$, the inclusion $L\subseteq \Lambda$, the equality $U\cap
K=\varnothing$ and the fact that $U$ is balanced, we see that every
$x\in W$ can be written as $x=\lambda y$, where $y\in K$ and
$\lambda\in{\mathbb D}=\{z\in \K:|z|\leq 1\}$. Since both $K$ and
${\mathbb D}$ are compact, $Q=\{\lambda y:\lambda\in {\mathbb D},\
y\in K\}$ is a compact subset of $X$. Since $W\subseteq Q$, $W$ is a
non-empty open set with compact closure. Such a set exists
\cite{shifer} only if $X$ is finite dimensional. This contradiction
completes the proof.
\end{proof}

The following lemma is a particular case of Lemma~5.1 in \cite{88}.

\begin{lemma}\label{fd} Let $X$ be a complex topological vector
space such that $2\leq \dim X<\infty$. Then $X$ supports no
supercyclic strongly continuous linear semigroups.
\end{lemma}

\begin{lemma}\label{iden} Let $X$ be an infinite dimensional
topological vector space, $\lambda\in\K$, $t_0>0$ and
$\{T_t\}_{t\in\R_+}$ be a strongly continuous linear semigroup such
that $T_{t_0}=\lambda I$. Then $\{T_t\}_{t\in\R_+}$ is not
supercyclic.
\end{lemma}

\begin{proof} Let $x\in X\setminus\{0\}$. It suffices to show that
$x$ is not a supercyclic vector for $\{T_t\}_{t\in\R_+}$.

First, we consider the case $\lambda=0$. By the strong continuity,
there is $s>0$ such that $0\notin K=\{T_tx:t\in[0,s]\}$ and $K$ is a
compact subset of $X$. By Lemma~\ref{tv1}, $A=\{zT_tx:z\in\K,\
t\in[0,s]\}$ is nowhere dense in $X$. Take $n\in\N$ such that
$ns\geq t_0$. Since $T_{t_0}=0$ and $ns\geq t_0$, we have
$T_s^n=T_{ns}=0$. Then $Y=\overline{T_s(X)}\neq X$. In particular,
$Y$ is nowhere dense in $X$. Clearly, $T_tx\in Y$ whenever $t\geq
s$. Hence $\{zT_tx:t\in\R_+,\ z\in\K\}$ is contained in $A\cup Y$
and therefore is nowhere dense in $X$. Thus $x$ is not a supercyclic
vector for $\{T_t\}_{t\in\R_+}$.

Assume now that $\lambda\neq 0$. Then $T_{t_0n}x=\lambda^nx\neq 0$
for every $n\in\Z_+$. Hence each of the compact sets $K_n=\{T_t
x:t_0n\leq t\leq t_0(n+1)\}$ with $n\in\Z_+$ does not contain $0$.
By Lemma~\ref{tv1}, the sets $A_n=\{zT_tx:z\in\C,\ t_0n\leq t\leq
t_0(n+1)\}$ are nowhere dense in $X$. On the other hand, for every
$t\in[t_0n,t_0(n+1)]$, $T_{t+t_0}x=T_tT_{t_0}x=\lambda T_t x$ and
therefore $A_n=A_{n+1}$ for each $n\in\Z_+$. Hence
$\{zT_tx:t\in\R_+,\ z\in\K\}$, which is clearly the union of $A_n$,
coincides with $A_1$ and therefore is nowhere dense. Thus $x$ is not
a supercyclic vector for $\{T_t\}_{t\in\R_+}$.
\end{proof}

\begin{lemma}\label{spa} Let $X$ be a complex
topological vector space and $\{T_t\}_{t\in\R_+}$ be a supercyclic
strongly continuous linear semigroup on $X$. Let also $t_0>0$ and
$\lambda\in\C$. Then the space $Y=\overline{(T_{t_0}-\lambda I)(X)}$
either coincides with $X$ or is a closed hyperplane in $X$.
\end{lemma}

\begin{proof} Using the semigroup property, it is easy to see
that $Y$ is invariant for each $T_t$. Factoring $Y$ out, we arrive
to a supercyclic strongly continuous linear semigroup
$\{S_t\}_{t\in\R_+}$ acting on $X/Y$, where $S_t(x+Y)=T_tx+Y$.
Obviously, $S_{t_0}=\lambda I$. If $X/Y$ is infinite dimensional, we
arrive to a contradiction with Lemma~\ref{iden}. If $X/Y$ is finite
dimensional and $\dim X/Y\geq 2$, we obtain a contradiction with
Lemma~\ref{fd}. Thus $\dim X/Y\leq 1$, as required.
\end{proof}

\begin{proof}[Proof of Proposition~$\ref{dich}$] Assume that there is
$t>0$ and $\lambda\in\K$ such that $(T_t-\lambda I)(X)$ is not dense
in $X$. By Lemma~\ref{spa}, $H=\overline{(T_t-\lambda I)(X)}$ is a
closed hyperplane in $X$. It is easy to see that $H$ is invariant
for every $T_t$.
\end{proof}

The following lemma provides some extra information on the second
case in Proposition~\ref{dich}.

\begin{lemma}\label{dich1} Let $X$ be a complex topological vector
space and $\{T_t\}_{t\in\R_+}$ be a strongly continuous linear
semigroup on $X$. Assume also that there is a closed hyperplane $H$
in $X$ such that $T_t(H)\subseteq H$ for every $t\in\R_+$ and let
$f\in X'$ be such that $H=\ker f$. Then there exists $w\in\C$ such
that $e^{wt}T'_tf=f$ for every $t\in\R_+$.
\end{lemma}

\begin{proof}Since $H=\ker f$ is invariant for every $T_t$, there is
a unique function $\phi:\R_+\to\C$ such that $T'_tf=\phi(t)f$ for
every $t\in\R_+$. Pick $u\in X$ such that $f(u)=1$. Then
$(T'_tf)(u)=f(T_tu)=\phi(t)$ for every $t\in\R_+$. Since
$\{T_t\}_{t\in\R_+}$ is strongly continuous, $\phi$ is continuous.
The semigroup property for $\{T_t\}_{t\in\R_+}$ implies the
semigroup property for the dual operators: $T'_0=I$ and
$T'_{t+s}=T'_tT'_s$ for every $t,s\in\R_+$. Together with the
equality $T'_tf=\phi(t)f$, it implies that $\phi(0)=1$ and
$\phi(t+s)=\phi(t)\phi(s)$ for every $t,s\in\R_+$. The latter and
the continuity of $\phi$ means that there is $w\in\C$ such that
$\phi(t)=e^{-wt}$ for each $t\in\R_+$. Thus $e^{wt}T'_tf=f$ for
$t\in\R_+$, as required.
\end{proof}

\section{Supercyclicity versus universality of affine maps}

In this section we relate the supercyclicity of an operator or a
semigroup in the case of the existence of an invariant hyperplane
and the universality of an affine map or an affine semigroup. We
start with the following general lemma.

\begin{lemma}\label{TRA} Let $X$ be a topological vector space, $u\in
X$, $f\in X'\setminus\{0\}$, $f(u)=1$ and $H=\ker f$. Assume also
that $\{T_a\}_{a\in A}$ is a family of continuous linear operators
on $X$ such that $T'_af=f$ for each $a\in A$. Then the family
$\F=\{zT_a:z\in\K,\ a\in A\}$ is universal if and only if the family
${\cal G}=\{R_a\}_{a\in A}$ of affine maps $R_a:H\to H$,
$R_ax=(T_au-u)+T_ax$ is universal on $H$. Moreover, $x\in X$ is
universal for $\F$ if and only if $x=\lambda(u+w)$, where
$\lambda\in \K\setminus\{0\}$ and $w$ is universal for $\cal G$.
Next, if $A=\Z_+$ and $T_a=T_1^a$ for every $a\in\Z_+$, then
$R_a=R_1^a$ for every $a\in\Z_+$. Finally, if $A=\R_+$ and
$\{T_a\}_{a\in\R_+}$ is a strongly $($respectively, jointly$)$
continuous linear semigroup, then $\{R_a\}_{a\in\R_+}$ is a strongly
$($respectively, jointly$)$ continuous affine semigroup.
\end{lemma}

\begin{proof} Since $T_a(H)\subseteq H$ for every $a$, vectors from
$H$ can not be universal for $\F$. Obviously, they also do not have
the form $\lambda(u+w)$ with $\lambda\in \K\setminus\{0\}$ and $w\in
H$.

Now let $x_0\in X\setminus H$. Then $f(x_0)\neq 0$ and therefore
$x=\frac{x_0}{f(x_0)}\in u+H$. Since $T_a(u+H)\subseteq u+H$ for
every $a\in A$, $O=\{T_ax:a\in A\}\subseteq u+H$. It is
straightforward to see that $x_0$ is universal for $\F$ if and only
if $O$ is dense in $u+H$. That is, $x_0$ is universal for $\F$ if
and only if $x$ is universal for the family $\{Q_a\}_{a\in A}$,
where each $Q_a:u+H\to u+H$ is the restriction of $T_a$ to the
invariant subset $u+H$. Obviously, the translation map $\Phi:H\to
u+H$, $\Phi(y)=u+y$ is a homeomorphism and $R_a=\Phi^{-1}Q_a\Phi$
for every $a\in A$. It follows that $x_0$ is universal for $\F$ if
and only if $\Phi^{-1}x=x-u$ is universal for $\cal G$. Denoting
$w=x-u$, we see that the latter happens if and only if
$x_0=f(x_0)(u+w)$ with $w\in\uu({\cal G})$.

Since $Q_a$ are the restrictions of $T_a$ to the invariant subset
$u+H$ and $R_a$ are similar to $Q_a$ with the similarity independent
on $a$, $\{R_a\}$ inherits all the semigroup or continuity
properties from $\{T_a\}$. The proof is complete.
\end{proof}

The following two lemmas are particular cases of Lemma~\ref{TRA}.

\begin{lemma}\label{trar} Let $X$ be a topological vector space, $u\in
X$, $f\in X'\setminus\{0\}$, $f(u)=0$ and $H=\ker f$. Then $T\in
L(X)$ satisfying $T'f=f$ is supercyclic if and only if the map
$R:H\to H$, $Rx=(Tu-u)+Tx$ is universal. Moreover, $x\in X$ is a
supercyclic vector for $T$ if and only if $x=\lambda(u+w)$, where
$\lambda\in \K\setminus\{0\}$ and $w\in\uu(R)$.
\end{lemma}

\begin{lemma}\label{trar1} Let $X$ be a topological vector space, $u\in
X$, $f\in X'\setminus\{0\}$, $f(u)=1$ and $H=\ker f$. Then a
strongly $($respectively, jointly$)$ continuous linear semigroup
$\{T_t\}_{t\in\R_+}$ on $X$ satisfying $T'_tf=f$ for $t\in\R_+$ is
supercyclic if and only if the strongly $($respectively, jointly$)$
continuous affine semigroup $\{R_t\}_{t\in\R_+}$ on $H$ defined by
$R_tx=(T_tu-u)+T_tx$ is universal. Moreover, $x\in X$ is a
supercyclic vector for $\{T_t\}_{t\in\R_+}$ if and only if
$x=\lambda(u+w)$, where $\lambda\in \K\setminus\{0\}$ and
$w\in\uu(\{R_t\}_{t\in\R_+})$.
\end{lemma}

\section{Universality of affine semigroups}

The proof of the following lemma is a matter of an easy routine
verification.

\begin{lemma}\label{uaff0} Let $X$ be a topological vector space,
$\{T_t\}_{t\in\R_+}$ be a collection of continuous affine maps on
$X$, $\{S_t\}_{t\in\R_+}$ be a collection of continuous linear
operators on $X$ and $t\mapsto w_t$ be a map from $\R_+$ to $X$ such
that $T_tx=w_t+S_tx$ for every $t\in\R_+$ and $x\in X$.

Then $\{T_t\}_{t\in\R_+}$ is an affine semigroup if and only if
$\{S_t\}_{t\in\R_+}$ is a linear semigroup,
\begin{equation}\label{afff}
\text{$w_0=0$ and}\ \ \ w_{t+s}=w_t+S_tw_s\ \ \ \text{for every
$s,t\in\R_+$.}
\end{equation}
Moreover, the semigroup $\{T_t\}_{t\in\R_+}$ is strongly continuous
if and only if $\{S_t\}_{t\in\R_+}$ is strongly continuous and the
map $t\mapsto w_t$ is continuous. Finally, the semigroup
$\{T_t\}_{t\in\R_+}$ is jointly continuous if and only if
$\{S_t\}_{t\in\R_+}$ is jointly continuous and the map $t\mapsto
w_t$ is continuous.
\end{lemma}

\begin{lemma}\label{fgfg} Let $X$ be a topological vector space
and $\{T_t\}_{t\in\R_+}$ be a universal strongly continuous affine
semigroup on $X$. Then $(I-T_t)(X)$ is dense in $X$ for every $t>0$.
\end{lemma}

\begin{proof} Assume the contrary. Then there is $s>0$ such that
$Y_0\neq X$, where $Y_0=\overline{(I-T_s)(X)}$. Let $Y$ be a
translation of $Y_0$, containing $0$: $Y=Y_0-u_0$ with $u_0\in Y_0$.
it is easy to see that, factoring out the closed linear subspace
$Y$, we arrive to the universal strongly continuous affine semigroup
$\{F_t\}_{t\in\R_+}$ on $X/Y$, where $F_t(x+Y)=T_tx+Y$ for every
$t\in\R_+$ and $x\in X$. By definition of $Y$, the linear part of
$F_s$ is $I$. Let $\alpha\in X/Y$ be a universal vector for
$\{F_t\}_{t\in\R_+}$. By Lemma~\ref{uaff0}, there is a strongly
continuous linear semigroup $\{G_t\}_{t\in\R_+}$ on $X/Y$ and a
continuous map $t\mapsto \gamma_t$ from $\R_+$ to $X/Y$ such that
$\gamma_0=0$, $F_t\beta=G_t\beta+\gamma_t$ and
$\gamma_{r+t}=\gamma_r+G_r\gamma_t=\gamma_t+G_t\gamma_r$ for every
$\beta\in X/Y$ and $r,t\in\R_+$. Using these relations and the
equality $G_s=I$, we obtain that
$F_{t+ns}\alpha=F_t\alpha+n\gamma_s$ for every $n\in\Z_+$ and
$t\in\R_+$. It follows that
$$
\{F_t\alpha:t\in\R_+\}=K+\Z_+\gamma_s,\ \ \ \text{where
$K=\{F_t\alpha:t\in[0,s]\}$.}
$$
Since $\alpha$ is universal for $\{F_t\}_{t\in\R_+}$, by the last
display, $O=K+\Z_+\gamma_s$ is dense in $X/Y$. Since $O$ is closed
as a sum of a compact set and a closed set, $O=X/Y$. On the other
hand, $O$ does not contain $-c\gamma_s$ for any sufficiently large
$c>0$. This contradiction completes the proof.
\end{proof}

\begin{lemma}\label{trtrt} Let $X$ be a topological vector space, $x\in
X$, $s>0$ and $\{T_t\}_{t\in\R_+}$ be a universal affine semigroup
on $X$. Assume also that $T_tx=S_tx+w_t$, where $\{S_t\}_{t\in\R_+}$
strongly continuous linear semigroup on $X$ and $t\mapsto w_t$ is a
continuous map from $\R_+$ to $X$. Then $\{S_t\}_{t\in\R_+}$ is
hypercyclic. Moreover, $\uu(\{S_t\}_{t\in\R_+})\cap
(w_s+(I-S_s)(X))\neq\varnothing$ for every $s>0$.
\end{lemma}

\begin{proof} Let $x\in \uu(\{T_t\}_{t\in\R_+})$ and fix $s>0$. By
Lemma~\ref{fgfg}, $(T_s-I)(X)$ is dense in $X$. Hence
$O=\{(T_s-I)T_tx:t\in\R_+\}$ is dense in $X$. Using the semigroup
property of $\{T_t\}_{t\in\R_+}$ and $\{S_t\}_{t\in\R_+}$ together
with (\ref{afff}), we get
$$
(T_s-I)T_tx=S_sS_tx+S_sw_t+w_s-S_tx-w_t=S_tS_sx+S_tw_s-
S_tx=S_t(w_s-(I-S_s)x)
$$
for every $t\in\R_+$. By the above display, $O$ is exactly the
$S_t$-orbit of $w_s-(I-S_s)x$. Since $O$ is dense in $X$,
$w_s-(I-S_s)x\in w_s+(I-S_s)(X)$ is a hypercyclic vector for
$\{S_t\}_{t\in\R_+}$ and therefore $\uu(\{S_t\}_{t\in\R_+})\cap
(w_s+(I-S_s)(X))\neq\varnothing$.
\end{proof}

\begin{lemma}\label{commu} Let $X$ be a topological vector space
and $\{T_t\}_{t\in\R_+}$ be an affine semigroup on $X$. Then for
every $t_1,\dots,t_n\in\R_+$ and every $z_1,\dots,z_n\in\K$
satisfying $z_1+{\dots}+z_n=1$, the map
$S=z_1T_{t_1}+{\dots}+z_nT_{t_n}$ commutes with every $T_t$.
\end{lemma}

\begin{proof}It is easy to verify that for every affine map $A:X\to
X$ and every $x_1,\dots,x_n\in X$,
\begin{equation*}
A(z_1x_1+{\dots}+z_nx_n)=z_1Ax_1+{\dots}+z_nAx_n\ \ \text{provided
$z_j\in\K$ and $z_1+{\dots}+z_n=1$.}
\end{equation*}
Let $t\in\R_+$. By the above display,
$T_tS=z_1T_tT_{t_1}+{\dots}+z_nT_tT_{t_n}$. Since $T_\tau$ commute
with each other, we get
$T_tS=z_1T_{t_1}T_t+{\dots}+z_nT_{t_n}T_t=ST_t$.
\end{proof}

\begin{lemma}\label{manyuni} Let $X$ be a topological vector space,
$\{T_t\}_{t\in\R_+}$ be a universal strongly continuous affine
semigroup on $X$ and $x\in\uu(\{T_t\}_{t\in\R_+})$. Then
$\Lambda(x)\subseteq \uu(\{T_t\}_{t\in\R_+})$, where
\begin{equation}\label{LAM}
\Lambda(x)=\{z_1T_{t_1}x+{\dots}+z_nT_{t_n}x:n\in\N,\ t_j\in\R_+,\
z_j\in\K,\ z_1+{\dots}+z_n=1\}.
\end{equation}
\end{lemma}

\begin{proof} Let $n\in\N$, $t_1,\dots,t_n\in\R_+$,
$z_1,\dots,z_n\in\K$ and $z_1+{\dots}+z_n=1$. We have to show that
$Ax\in\uu(\{T_t\}_{t\in\R_+})$, where
$A=z_1T_{t_1}+{\dots}+z_nT_{t_n}$. By Lemma~\ref{commu}, $A$
commutes with each $T_t$. Since $x\in\uu(\{T_t\}_{t\in\R_+})$, it
suffices to verify that $A(X)$ is dense in $X$. By
Lemma~\ref{uaff0}, we can write $T_ty=S_ty+w_t$ for every $y\in X$,
where $\{S_t\}_{t\in\R_+}$ is a strongly continuous linear semigroup
on $X$ and $t\mapsto w_t$ is a continuous map from $\R_+$ to $X$. By
Lemma~\ref{trtrt}, $\{S_t\}_{t\in\R_+}$ is hypercyclic. As shown in
\cite{semi}, every non-trivial linear combination of members of a
hypercyclic strongly continuous linear semigroup has dense range.
Thus $B=z_1S_{t_1}+{\dots}+z_nS_{t_n}$ has dense range. Since $A(X)$
is a translation of $B(X)$, $A(X)$ is also dense in $X$, which
completes the proof.
\end{proof}

\begin{proof}[Proof of Theorem~$\ref{BBB}$] Let $X$ be a topological
vector space and $\{T_t\}_{t\in\R_+}$ be a universal jointly
continuous affine semigroup on $X$. Lemmas~\ref{uaff0}
and~\ref{trtrt} provide a hypercyclic jointly continuous linear
semigroup on $X$. By Theorem~A, there is a hypercyclic continuous
linear operator on $X$. Since no such thing exists on a finite
dimensional topological vector space \cite{ww}, $X$ is infinite
dimensional. Since any compact subspace of an infinite dimensional
topological vector space is nowhere dense \cite{shifer}, condition
(1) of Proposition~\ref{gege} is satisfied. Now let
$x\in\uu(\{T_t\}_{t\in\R_+})$. By Lemma~\ref{manyuni}, the set
$\Lambda(x)$ defined in (\ref{LAM}) consists entirely of universal
vectors for $\{T_t\}_{t\in\R_+}$. Clearly,
$\{T_tx:t\in\R_+\}\subseteq\Lambda(x)$. By its definition,
$\Lambda(x)$ is an affine subspace (=a translation of a linear
subspace) of $X$. Since every affine subspace of a topological
vector space is connected, locally path connected and simply
connected, $\Lambda(x)$ satisfies all requirements for the set
$Y_{c,x}$ (for every $c>0$) from condition (2) in
Proposition~\ref{gege}. By Proposition~\ref{gege},
$\uu(T_s)=\uu(\{T_t\}_{t\in\R_+})$ for every $s>0$, as required.
\end{proof}

\section{Proof of Theorem~\ref{main}}

Let $X$ be a complex topological vector space and
$\{T_t\}_{t\in\R_+}$ be a supercyclic jointly continuous linear
semigroup on $X$. We have to prove that each $T_s$ with $s>0$ is
supercyclic and the set of supercyclic vectors of $T_s$ coincides
with the set of supercyclic vectors of $\{T_t\}_{t\in\R_+}$. If
$T_t-\lambda I$ has dense range for every $t>0$ and every
$\lambda\in\C$, then Proposition~C provides the required result.
Otherwise, by Proposition~\ref{dich}, there is a closed hyperplane
$H$ in $X$ invariant for every $T_t$. By Lemma~\ref{dich1}, there
are $f\in X'$ and $\alpha\in\C$ such that $H=\ker f$ and $e^{\alpha
t}T'_tf=f$ for every $t\in\R_+$. Clearly $\{e^{\alpha
t}T_t\}_{t\in\R_+}$ is a jointly continuous supercyclic linear
semigroup on $X$ with the same set ${\cal S}$ of supercyclic vectors
as the original semigroup $\{T_t\}_{t\in\R_+}$. Fix $u\in X$
satisfying $f(u)=1$. Now fix $s>0$ and $v\in{\cal S}$. We have to
show that $v$ is supercyclic for $T_s$. By Lemma~\ref{trar1},
applied to the semigroup $\{e^{\alpha t}T_t\}_{t\in\R_+}$, we can
write $v=\lambda(u+y)$, where $\lambda\in\K\setminus\{0\}$ and $y$
is a universal vector for the jointly continuous affine semigroup
$\{R_t\}_{t\in\R_+}$ on $H$ defined by the formula
$R_tx=w_t+e^{\alpha t}T_tx$ with $w_t=(e^{\alpha t}T_t-I)u$. By
Theorem~\ref{BBB}, $y$ is universal for $R_s$. By Lemma~\ref{trar},
$v=\lambda(u+y)$ is a supercyclic vector for $e^{\alpha s}T_s$ and
therefore $v$ is a supercyclic vector for $T_s$. The proof is
complete.

\section{Remarks}

By Lemma~\ref{trtrt}, universality of a strongly continuous affine
semigroup implies hypercyclicity of the underlying linear semigroup.
The following example shows that that the converse is not true.

\begin{example}\label{fiii} Consider the backward weighted
shift $T\in L(\ell_2)$ with the weight sequence
$\{e^{-2n}\}_{n\in\N}$. That is, $Te_0=0$ and $Te_n=e^{-2n}e_{n-1}$
for $n\in\N$, where $\{e_n\}_{n\in\Z_+}$ is the standard basis of
$\ell_2$. Then the jointly continuous linear semigroup
$\{S_t\}_{t\in\R_+}$ with $S_t=e^{t\ln(I+T)}$ is hypercyclic.
Moreover, there exists a continuous map $t\mapsto w_t$ from $\R_+$
to $\ell_2$ such that $\{T_t\}_{t\in\R_+}$ is a jointly continuous
non-universal affine semigroup, where $T_tx=w_t+S_tx$ for
$x\in\ell_2$.
\end{example}

\begin{proof} Since $T$, being a compact weighted backward shift, is
quasinilpotent, the operator $\ln(I+T)$ is well defined and bounded
and $\{S_t\}_{t\in\R_+}$ is a jointly continuous linear semigroup.
Moreover, $S_1=I+T$ is hypercyclic according to Salas \cite{salas}
as a sum of the identity operator and a backward weighted shift.
Hence  $\{S_t\}_{t\in\R_+}$ is hypercyclic.

Let $u\in \ell_2$, $u_n=(n+1)^{-1}$ for $n\in\Z_+$. For each
$t\in\R_+$, let $w_t=\nu_t(T)u$, where
$\nu_s(z)=\sum\limits_{n=1}^\infty
\frac{s(s-1)\dots(s-n+1)}{n!}z^{n-1}$. Since $T$ is quasinilpotent,
$\nu_t(T)$ are well defined bounded linear operators and the map
$t\mapsto \nu_t(T)$ is operator-norm continuous. Hence $t\mapsto
w_t$ is continuous as a map from $\R_+$ to $\ell_2$. It is easy to
verify that $w_0=0$, $w_1=u$ and $w_{t+s}=S_tw_s+w_t$ for every
$s,t\geq 0$. By Lemma~\ref{uaff0}, $\{T_t\}_{t\in\R_+}$ is a jointly
continuous affine semigroup, where $T_tx=w_t+S_tx$. It remains to
show that $\{T_t\}_{t\in\R_+}$ is non-universal. Assume the
contrary. Since $w_1=u$ and $S_1=I+T$, Lemma~\ref{trtrt} implies
that the coset $u+T(\ell_2)$ must contain a hypercyclic vector for
$I+T$. This however is not the case as shown in
\cite[Proposition~7.16]{66}.
\end{proof}

Recall that a topological space $X$ is called a {\it Baire space} if
the intersection of any countable collection of dense open subsets
of $X$ is dense in $X$.

\begin{remark}\label{re1} Let $X$ be a topological vector space and $S\in
L(X)$ be hypercyclic. If $u\in (I-S)(X)$, then the affine map
$Tx=u+Sx$ is universal. Indeed, let $w\in X$ be such that $u=w-Sw$.
It is easy to show that $T^nx=w+S^n(x-w)$ for every $x\in X$ and
$n\in\N$. Thus $x$ is universal for $T$ if and only if $x-w$ is
universal for $S$.

If additionally $X$ is separable metrizable and Baire, then a
standard Baire category type argument shows that the set of $u\in X$
for which the affine map $Tx=u+Sx$ is universal is a dense
$G_\delta$-subset of $X$. Example~\ref{fiii} shows that this set can
differ from $X$.
\end{remark}

Recall that a locally convex topological vector space $X$ is called
{\it barrelled} if every closed convex balanced subset $B$ of $X$
satisfying $X=\bigcup\limits_{n=1}^\infty nB$ contains a
neighborhood of 0. As we have already mentioned in the introduction,
the joint continuity of a linear semigroup follows from the strong
continuity if the underlying space $X$ is an $\F$-space. The same is
true for wider classes of topological vector spaces. For instance,
it is sufficient for $X$ to be a Baire topological vector space or a
barrelled locally convex topological vector space \cite{shifer}.
Thus the following observation holds true.

\begin{remark}\label{j-s}
The joint continuity condition in Theorems~A, \ref{main}
and~\ref{BBB} can be replaced by the strong continuity, provided $X$
is Baire or $X$ is locally convex and barrelled.
\end{remark}

For general topological vector spaces however strong continuity of a
linear semigroup does not imply joint continuity. Moreover, the
following example shows that Theorem~A fails in general if the joint
continuity condition is replaced by the strong continuity. Recall
that the Fr\'echet space $L^2_{\rm loc}(\R_+)$ consists of the
(equivalence classes of) scalar valued functions $\R_+$, square
integrable on $[0,c]$ for each $c>0$. Its dual space can be
naturally interpreted as the space $L^2_{\rm fin}(\R_+)$ of
(equivalence classes of) square integrable scalar valued functions
$\R_+$ with bounded support. The duality between $L^2_{\rm
loc}(\R_+)$ and $L^2_{\rm fin}(\R_+)$ is provided by the natural
dual pairing $\langle f,g\rangle =\int_0^\infty f(t)g(t)\,dt$.
Obviously the linear semigroup $\{S_t\}_{t\in\R_+}$ of backward
shifts $S_tf(x)=f(x+t)$ is strongly continuous and therefore jointly
continuous on the Fr\'echet space $L^2_{\rm loc}(\R_+)$. It follows
that the same semigroup is strongly continuous on $L^2_{\sigma,\rm
loc}(\R_+)$ being $L^2_{\rm loc}(\R_+)$ endowed with the weak
topology.

\begin{example}\label{AAA} Let $X=L^2_{\sigma,\rm
loc}(\R_+)$ and $\{S_t\}_{t\in\R_+}$ be the above strongly
continuous semigroup on $X$. Then there is $f\in X$ hypercyclic for
$\{S_t\}_{t\in\R_+}$ such that $f$ is non-hypercyclic for $S_1$.
\end{example}

\begin{proof} Let $H$ be the hyperplane in $L^2[0,1]$ consisting of
the functions with zero Lebesgue integral. Fix a norm-dense
countable subset $A$ of $H$. One can easily construct $f\in L^2_{\rm
loc}(\R_+)$ such that
\begin{itemize}
\item[(a)]for every $n\in\N$, the function $f_n:[0,1]\to\K$,
$f_n(t)=f(n+t)$ belongs to $A;$
\item[(b)]for every $n\in\N$ and $h_1,\dots,h_n\in A$, there is
$m\in\N$ such that $h_j=f_{m+j}$ for $1\leq j\leq n$.
\end{itemize}

For $s\in\R_+$, let $\chi_s\in X'=L^2_{\rm fin}(\R_+)$ be the
indicator function of the interval $[s,s+1]$: $\chi_s(t)=1$ if
$s\leq t\leq s+1$ and $\chi_s(t)=0$ otherwise.  By (a),
$S_1^nf\in\ker\chi_0$ for every $n\in\N$ and therefore $f$ is not a
hypercyclic vector for $S_1$.

It remains to show that $f$ is a hypercyclic vector for
$\{S_t\}_{\in\R_+}$ acting on $X$. Using (a) and (b), we see that
the Fr\'echet space topology closure of the orbit
$\{S_tf:t\in\R_+\}$ is exactly the set
$$
O=\bigcup_{s\in[0,1)}\bigcap_{n\in\Z_+}\ker\chi_{s+n}.
$$
In order to show that $f$ is hypercyclic for $\{S_t\}_{\in\R_+}$
acting on $X$, it suffices to verify that $O$ is dense in
$L^2_{\sigma,\rm loc}(\R_+)$. Assume the contrary. Then there is a
weakly open set $W$ in $L^2_{\rm loc}(\R_+)$, which does dot
intersect $O$. That is, there are linearly independent
$\phi_1,\dots,\phi_m\in L^2_{\rm fin}(\R_+)$ and
$c_1,\dots,c_m\in\K$ such that
$$
\text{$\max\limits_{1\leq j\leq m}|c_j-\langle g,\phi_j\rangle|\geq
1$ for every $g\in O$.}
$$

Let $k\in\N$ be such that each $\phi_j$ vanishes on $[k,\infty)$.
Pick any $0<t_0<{\dots}<t_{m}<1$. Note that for every
$l\in\{0,\dots,m\}$, the restrictions of the functionals $\phi_j$ to
$\bigcap\limits_{n=0}^k\ker\chi_{t_l+n}$ are not linearly
independent. Indeed, otherwise we can find
$h_0\in\bigcap\limits_{n=0}^k\ker\chi_{t_l+n}$ such that $\langle
h_0,\phi_j\rangle=c_j$ for $1\leq j\leq m$. It is easy to see that
there is $h\in L^2_{\rm loc}(\R_+)$ such that
$h\bigr|_{[0,k]}=h_0\bigr|_{[0,k]}$, $h\bigr|_{[k+1,\infty)}=0$ and
$\langle h,\chi_{t_l+k-1}\rangle=\langle h,\chi_{t_l+k}\rangle=0$.
Then $\langle h,\phi_j\rangle=c_j$ for $1\leq j\leq m$ and $h\in
\bigcap\limits_{n=0}^\infty\ker\chi_{t_l+n}\subseteq O$. We have
arrived to a contradiction with the above display.

The fact that $\phi_j$ are not linearly independent on
$\bigcap\limits_{n=0}^k\ker\chi_{t_l+n}$ implies that there is a
non-zero $g_l\in\spann\{\phi_1,\dots,\phi_m\}\cap
\spann\{\chi_{t_l},\dots,\chi_{t_l+k}\}$. Since $\chi_{t_l+r}$ are
all linearly independent, $g_0,\dots,g_m$ are $m+1$ linearly
independent vectors in the $m$-dimensional space
$\spann\{\phi_1,\dots,\phi_m\}$. This contradiction completes the
proof.
\end{proof}


\small\rm

\vskip1truecm

\scshape

\noindent Stanislav Shkarin

\noindent Queens's University Belfast

\noindent Pure Mathematics Research Centre

\noindent University road, Belfast, BT7 1NN, UK

\noindent E-mail address: \qquad {\tt s.shkarin@qub.ac.uk}


\begin{thebibliography}{99}

\itemsep=-2pt

\bibitem{bama-book}F.~Bayart and E.~Matheron, \it Dynamics of linear
operators, \rm Cambridge University Press, 2009

\bibitem{ex2}L.~Bernal-Gonz\'alez and K.-G.~Grosse-Erdmann, \it Existence and
nonexistence of hypercyclic semigroups, \rm  Proc. Amer. Math. Soc.
\bf135\rm\ (2007), 755--766

\bibitem{semi}J.~Conejero, V.~M\"uller and A.~Peris, \it
Hypercyclic behaviour of operators in a hypercyclic $C\sb
0$-semigroup, \rm J. Funct. Anal. \bf244\rm\ (2007), 342--348

\bibitem{salas}H.~Salas, \it Hypercyclic weighted shifts, \rm Trans. Amer. Math.
Soc. \bf347\rm\ (1995), 993--1004

\bibitem{66}S.~Shkarin, \it Universal elements for non-linear
operators and their applications, \rm J. Math. Anal. Appl.
\bf348\rm\ (2008), 193--210

\bibitem{88}S.~Shkarin, \it Hypercyclic and mixing operator semigroups,
\rm  Proc. Edinb. Math. Soc. [to appear]

\bibitem{shifer}H.~Sch\"afer, \it Topological vector spaces, \rm
Springer, New York, 1971

\bibitem{ww}J.~Wengenroth, \it Hypercyclic operators on non-locally
convex spaces, \rm Proc. Amer. Math. Soc. \bf131\rm\ (2003),
1759--1761


\end{thebibliography}
\end{document}